\documentclass[aoas,preprint]{imsart}

\RequirePackage[OT1]{fontenc}
\RequirePackage{amsthm,amsmath,amssymb,enumerate}
\RequirePackage[authoryear]{natbib}
\RequirePackage[colorlinks,citecolor=blue,urlcolor=blue]{hyperref}
\usepackage{graphicx}
\usepackage{color}

\startlocaldefs
\numberwithin{equation}{section}
\theoremstyle{plain}
\newtheorem{thm}{Theorem}
\newtheorem{lemma}{Lemma}
				
\newtheorem{proposition}{Proposition}
\theoremstyle{definition}
\newtheorem{definition}{Definition}
\newtheorem{assumption}{Assumption}

\def\one{{\bf 1}\hskip-.5mm}

\def\E{{\mathbb E}}
\def\P{{\mathbb P}}
\def\R{{\mathbb R}}
\def\Z{{\mathbb Z}}
\def\P {\mathbb{P}}

\def\N{{\mathbb N}}

\def\cF {\mathcal{F}}

\newcommand\floor[1]{\lfloor#1\rfloor}

\endlocaldefs

\begin{document}

\begin{frontmatter}
\title{Estimating the interaction graph of stochastic neural dynamics}
\runtitle{Estimating the interaction graph for neural dynamics}

\begin{aug}
\author{\fnms{A.} \snm{Duarte}\thanksref{m1}\ead[label=e1]{aline.duart@gmail.com}},
\author{\fnms{A.} \snm{Galves}\thanksref{m1}\ead[label=e2]{galves@usp.br}},
\author{\fnms{E.} \snm{L\"ocherbach}\thanksref{m2}\ead[label=e3]{eva.loecherbach@u-cergy.fr}}
\and
\author{\fnms{G.} \snm{Ost}\thanksref{m1}
\ead[label=e4]{guilhermeost@gmail.com}
\ead[label=u1,url]{http://www.foo.com}}

\runauthor{A. Duarte et al.}

\affiliation{Universidade de S\~ao Paulo\thanksmark{m1} and Universit\'e de Cergy-Pontoise\thanksmark{m2}}

\address{
E. L\"ocherbach\\
Universit\'e de Cergy-Pontoise\\ 
AGM, CNRS-UMR 8088,\\
95000 Cergy-Pontoise,\\
France\\
\printead{e3}}

\address{
A. Duarte\\
A. Galves \\
G. Ost\\
Universidade de S\~ao Paulo\\
Instituto de Matem\'atica e Estat\'istica\\
S\~ao Paulo, Brazil\\
\printead{e1}\\
\printead{e2}\\
\printead{e4}
}

\end{aug}

\begin{center}
{\it Dedicated to Enza Orlandi, in memoriam}
\end{center}
\begin{abstract}
In this paper we address the question of statistical model selection for a class of  stochastic models of biological neural nets.  Models in this class are systems of interacting chains with memory of variable length. Each chain describes the activity of a single neuron, indicating whether it spikes or not at a given time.  The spiking probability of a given neuron depends on the time evolution of its {\it presynaptic neurons} since its last spike time. 
When a neuron spikes, its potential is reset to a resting level and postsynaptic current pulses are generated, modifying the membrane potential of all its \textit{postsynaptic neurons}. The relationship between a neuron and its pre- and postsynaptic neurons defines an oriented graph, the {\it interaction graph} of the model.  The goal of this paper is to estimate this graph based on the  observation of the spike activity of a finite set of neurons over a finite time. 
We provide explicit exponential upper bounds for the probabilities of under- and overestimating the interaction graph restricted to the observed set and obtain the strong consistency of the estimator. 
Our result does not require stationarity nor uniqueness of the invariant measure of the process.
\end{abstract}



\end{frontmatter}

\section{Introduction}
This paper addresses the question of statistical model selection for a class of stochastic processes describing biological neural networks.  The activity of the neural net is described by a countable system of interacting chains with memory of variable length representing the spiking activity of the different neurons.   The interactions between neurons are defined in terms of their interaction neighborhoods. The interaction neighborhood of a neuron is given by the set of all its presynaptic neurons. We introduce a statistical selection procedure in this class of stochastic models to estimate the interaction neighborhoods.

The stochastic neural net we consider can be described as follows. Each neuron spikes with a probability which is an increasing function of its membrane potential.  The membrane potential of a given neuron depends on the accumulated spikes coming from the presynaptic neurons since its last spike time. 
When a neuron spikes, its potential is reset to a resting level and at
the same time postsynaptic current pulses are generated, modifying the membrane potential of all its postsynaptic neurons.

Recently, several papers have been devoted to the probabilistic study of these models, starting with \cite{GalEva:13} who provided a rigorous mathematical framework to study such systems with an infinite number of interacting components, evolving in discrete time. Its continuous time version has been subsequently studied in \cite{Errico:14}, \cite{AG:14}, \cite{Evafou:14},  \cite{RobTou:14}, \cite{PierreEva:14}, \cite{Duarte_Ost:15} and \cite{Karina:15}. 
We also refer to \cite{Brochinietal:16} and the references therein for a simulation study and mean field analysis.
All these papers deal with probabilistic aspects of the model, not with statistical model selection. 

Statistical model selection for graphical models has been largely discussed in the literature.
Recently much effort has been devoted to
estimating the graph of interactions underlying e.g.\ finite volume Ising models (\cite{ravikumar:10}, \cite{Montanari:09}, \cite{BreElcSly:08} and \cite{Bresler:15}), infinite volume Ising models (\cite{GalOrlTak:15}, \cite{lerTak:11} and \cite{lerTak:16}),  Markov random fields (\cite{CT} and \cite{Talata:14}) and variable-neighborhood random fields (\cite{LochOrla:11}).

Graphical models are  very interesting from a mathematical point of view. However, their application to the stochastic modeling of neural data has a major drawback, namely the assumption that the configuration describing the neural activity at a given time follows a Gibbsian distribution. To the best of our knowledge, this Gibbsian assumption is not supported by biological considerations of any kind. 

Statistical methods for selecting the graph of interactions in neural nets start probably with \cite{Brillinger1979} and \cite{brillinger:88}. Recently, new approaches have been proposed by \cite{reynaud:13} in the framework of  multivariate point processes and Hawkes processes. Let us also mention   \cite{plosof}  where a new experimental design is introduced and studied from a numerical point of view. 
These articles have the following drawback. The firstly mentioned ones only consider finite systems of neurons or processes having a fixed finite memory, the lastly cited does only propose a numerical study.

The present paper provides a rigorous mathematical approach to the problem of inference in neural dynamics. Its main features are the following. 

\begin{enumerate}
\item \label{item1}
 The processes we consider are not Markovian. They are systems of interacting chains with memory of variable length. 
\item \label{item2}
Our approach does not rely on any Gibbsian assumption. 
\item \label{item3}
We can deal with the case in which the system possesses several invariant measures.  
\item \label{item4}
Infinite systems of neurons can be also treated under suitable assumptions on the synaptic weights.  
\end{enumerate}

Let us briefly comment on Features \ref{item3} and \ref{item4}. Feature \ref{item3} makes our model suitable to describe long-term memory in which the asymptotic law of the system depends on its initial configuration.  This initial configuration can be seen as the effect of an external stimulus to which the system is exposed at the beginning of the experiment. In this perspective, the asymptotic distribution can be interpreted as the neural encoding of this stimulus by the brain. Feature \ref{item4} enables us to deal with arbitrarily high dimensional systems, taking into account the fact that the brain consists of a huge (about $10^{11}$) number of interacting neurons. 

The models we consider can be seen as a version of the Integrate and
Fire (IF) model with random thresholds, but only in cases in which the
postsynaptic current pulses are of the exponential type. Indeed, only in such cases, the time evolution of the family
of membrane potentials is a Markov process. For general postsynaptic current pulses, this is not true.

Therefore, we could say that our model is a non-Markovian version of the IF model with random thresholds. This fact places the class of models considered here within a  classical and widely accepted framework of modern neuroscience. Indeed, as pointed out by an anonymous referee,  IF models have a long and rich history, going back to the fundamental work \cite{HH}. For more insights on IF-models we refer the interested reader to classical textbooks such as \cite{dyan} and \cite{Gerstner:2002:SNM:583784}.

The inherent randomness of the thresholds in our model leads to random neuronal responses instead of deterministic  ones. The idea that the spike activity is intrinsically random  and not deterministic can be traced back to  \cite{adrian}, see also \cite{adrian2}. Under the name of ``escape noise''-models, this question has then been further emphasized by \cite{vanhemmen} and \cite{Gerstner95}.   

We conclude this introduction by briefly describing the statistical selection procedure we propose. We observe the process within a finite sampling region during a finite time interval. For each neuron $i$ in the sample, we estimate 
its spiking probability given the spike trains of all other neurons since its last spike time. For each neuron $j\neq i ,$ we then introduce a measure of sensibility of this conditional spiking probability with respect to changes within the spike train of neuron $ j.$ If this measure of sensibility is {\it statistically small}, we
conclude that neuron $j$ does not belong to the interaction neighborhood of neuron $i$. 

For this selection procedure, we give precise error bounds for the probabilities of under- and overestimating finite interaction neighborhoods implying the consistency of the procedure in Theorem \ref{thm:1}. For interaction neighborhoods which are not contained within the sampling region, a coupling between the process and its locally finite range approximation reduces the estimation problem to the situation of Theorem \ref{thm:1}. The coupling result is presented in Proposition \ref{thm:couplage}. In our proofs we rely on a new conditional Hoeffding-type inequality which is of independent interest, stated in Proposition \ref{prop:1}.   

We stress that, in our class of models, the probability of a neuron to spike depends only on the history of the process since its last spike time. Therefore, temporal dependencies do not need to be estimated, making our estimation problem different from classical context tree estimation as considered in \cite{CT} and \cite{GalvesLeonardi:08}.
We refer the reader to the companion paper by \cite{Brochinietal:17} where our statistical selection procedure is explored, tested and applied to the analysis of simulated and real neural activity data. 

The paper is organized as follows. In Section \ref{sec:def} we introduce the model and the selection procedure, present the main assumptions and formulate the main results, Theorems \ref{thm:1} and \ref{thm:main2}. In Section \ref{sec:dev.ineq}, we derive some exponential inequalities including a new conditional Hoeffding-type inequality, presented in Proposition \ref{prop:1}, which is interesting by itself. 
The proofs of Theorems \ref{thm:1} and \ref{thm:main2} are presented in Sections \ref{sec:proof.finite.case} and \ref{sec:proofs.infinite.case}  respectively. In section \ref{sec:proof.runtime} we discuss the time complexity of our selection procedure.   In Appendix \ref{sec:proba} we present the coupling result stated in Proposition \ref{thm:couplage}. 


\section{Model definition and main results}\label{sec:def}
\subsection{A stochastic model for a system of interacting neurons}  
Throughout this article,  $I$  denotes a countable set,  $W_{j \to i}\in \R$ with $i,j\in I$,  a collection of real numbers such that $W_{j \to j } = 0 $ for all $j,$ and for $i\in I$, $\varphi_i : \R  \to [ 0, 1 ]$ a  non-decreasing measurable function and $g_i=(g_i(t))_{t\geq 1}$ a sequence of strictly positive real numbers. 

In order to be consistent with the neuroscience terminology (see \cite{Gerstner:2002:SNM:583784}), we call 
\begin{itemize}
\item
$I$ the {\it set of neurons}, 
\item
$ W_{j \to i } $ the {\it synaptic weight of neuron $j$ on neuron $i,$} 
\item
$\varphi_i $ the  {\it spike rate function} of neuron $i$,
\item $g_i $ the {\it postsynaptic current pulse} of neuron i.
\end{itemize}

We consider a stochastic chain $(X_t )_{ t \in \Z }$ taking values in $ \{ 0, 1 \}^I ,$ defined on a suitable probability space $ ( \Omega , { \cal A} , \P ) .$ For each $i\in I$ and $t \in \Z,$  $ X_t (i) = 1 $, if neuron $i$ spikes at time $t$ and $X_t(i) = 0,$ otherwise. For each neuron $i \in I$ and each time $t \in \Z $, let
\begin{equation}
\label{def:0}
L_t^i = \sup \{ s \leq t  : X_s (i) = 1  \}, 
\end{equation} 
be the last spike time of neuron $i$ before time $t.$ Here, we adopt the convention that  $\sup \{ \emptyset  \}=-\infty.$ 

For each $t \in \Z $, we call $\cF_t$ the sigma algebra generated by the past events up to time $t $, that is, 
$$ \cF_t :=\sigma(X_s(j),s\leq t,j\in I ).$$
The stochastic chain $(X_t )_{ t \in \Z }$ is defined as follows. For each time $t \in \Z, $ for any finite set $ F \subset I $ and any choice $ a(i ) \in \{ 0, 1 \} , i \in F,$ 
\begin{equation}
\label{def:1}
\P(X_{t+1}(i) =a(i), i \in F |\cF_{t})=\prod_{i\in F}\P(X_{t+1}(i)=a(i) |\cF_{t}),
\end{equation}
where for each $i\in I$ and $t\in \Z,$
\begin{equation}
\label{def:2}
\P(X_{t+1}(i)=1|\cF_{t})= \varphi_i\Big( \sum_{j\in I}W_{j\to i}\sum_{s=L_{t}^i+1}^{t} g_j(t+1 -s)X_s(j)\Big),
\end{equation}
if $ L_{t}^i < t ,$ and 
$$
\P(X_{t+1}(i)=1|\cF_{t})= \varphi_i (0 ),
$$
otherwise.  

Suitable conditions ensuring the existence of a stochastic chain of this type are presented at the beginning of Section \ref{sec:2.3}.

\subsection{Neighborhood estimation procedure}
We write
$$ x = (x_t (i ) )_{-\infty < t \le 0, i \in I}$$ 
for any configuration $ x  \in \{0,1\}^{I\times \{\ldots,-1,0\}} ,$ and for any  $F \subset I, $ we  write
$$ x_t ( F) = (x_t ( i ) , i \in F ) .$$ 
Moreover, for any $x \in \{0,1\}^{I\times \{\ldots,-1,0\}} , $ 
$$X_{-\infty}^0 =  x, \mbox{  if $ X_t (i ) = x_t (i) $ for all $ - \infty < t \le 0 $ and for all $ i \in I.$}$$ 
Finally, for any $\ell\geq 1$, $t  \in \Z, $ $F \subset I $ and $w\in  \{ 0, 1\}^{ \{ - \ell , \ldots , -1\} \times F} ,$ we  write 
$$ X^{t-1}_{t- \ell } ( F) = w, \mbox{ if $ X_{t-s } (j ) = w_{-s} (j ) , $ for all $ 1 \le s \le \ell $ and for all $ j \in F$}$$
and 
$$ X_{t - \ell - 1 }^{t - 1 } (i ) = 1 0^\ell, \mbox{ if } X_{t -s} (i ) = 0 , \mbox{ for all } 1 \le s \le \ell  \ \mbox{and} \ X_{t - \ell - 1 } (i ) = 1 .$$ 
Throughout the article, $s, t \in \Z $ will be time indices, while $n \in \N$ will be saved for future use as the length of the time interval during which the neural network is observed.  

Let 
\begin{equation}
\label{def:Vi}
V_i=\{j\in I\setminus\{i\}: W_{j\to i}\neq 0\}
\end{equation}
be the set of presynaptic neurons of neuron $i$. The set $V_i$ is called the {\em interaction neighborhood} of neuron $i$.  
The goal of our statistical selection procedure is to identify the set  $V_{i}$ from the data in a consistent way. 

Let $X_1(F),\ldots, X_n(F)$ be a sample where $F\subset I$ is a finite sampling region and $n\geq 3$ is the length of the time interval during which the network has been observed. 
For any fixed $i \in F,$ we want to estimate its interaction neighborhood $V_i.$

\begin{figure}[h!]
\includegraphics[scale=0.35]{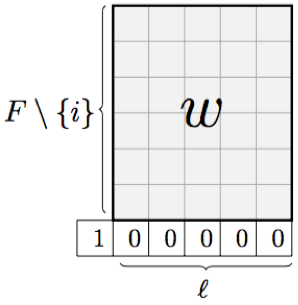}
\caption{Local past $w\in  \{ 0, 1\}^{ \{ - \ell , \ldots , -1\} \times F  \setminus \{i\}} $ outside of $i$ with $\ell=5$ and $|F|=7$.}
\label{fig:1}
\end{figure}

Our procedure is defined as follows. For each $1\leq \ell\leq n-2$, 
local past $w\in  \{ 0, 1\}^{ \{ - \ell , \ldots , -1\} \times F  \setminus \{i\}} $ outside of $i$ (see Figure \ref{fig:1}) and symbol $a\in\{0,1\}$, we define 
\begin{equation*}
N_{(i,n)}(w,a)=
\sum_{t=\ell+2}^{n}{\one}{\{ X_{t-\ell-1}^{t-1}(i)=10^{\ell},X_{t-\ell}^{t-1}(F \setminus\{i\})=w,X_t(i)=a\}}.
\end{equation*}
The random variable $N_{(i,n)}(w,a)$ counts the number of occurrences of $w$ followed or not by a spike of neuron $i$ ($a=1$ or $a=0,$ respectively) in the sample $X_1 (F) ,\ldots, X_n (F) ,$ when the last spike of neuron $i$ has occurred $\ell + 1$ time steps before in the past, see Figure \ref{fig:2}. 

\begin{figure}[h!]
\includegraphics[scale=0.35]{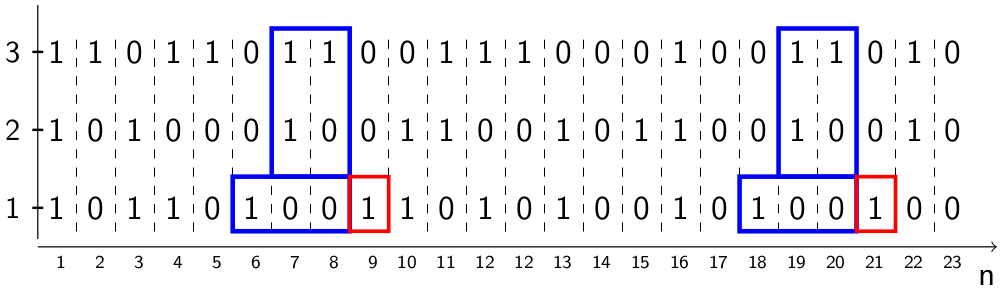}
\caption{Example for $N_{(i,n)}(w,1) = 2,$ where $i=1,$ for a given word $ w$ (in blue), $\ell=2,$ $|F|=3$ and $n = 23.$}
\label{fig:2}
\end{figure}

We define the empirical probability of neuron $i$ having a spike at the next step given $w$ by  
\begin{equation}
\label{def:trans.prob.emp}
\hat{p}_{(i,n)}(1|w)=\frac{N_{(i,n)}(w,1)}{ N_{(i,n)}(w)},
\end{equation}
when $N_{(i,n)}(w):=N_{(i,n)}(w, 0 ) +N_{(i,n)}(w, 1 )>0.$ 

For any fixed parameter $\xi\in (0,1/2)$, we consider the following set
\begin{equation}\label{eq:xi}
\mathcal{T}_{(i,n)}= \Big\{w\in  \bigcup_{\ell =1}^{n-2}\{ 0, 1\}^{ \{ - \ell , \ldots , -1\} \times F \setminus \{i\}} :N_{(i,n)}(w)\geq n^{1/2+\xi} \ \Big\}. 
\end{equation} 

We use the notation $|w| = \ell$ whenever $w\in \{ 0, 1\}^{ \{ - \ell , \ldots , -1\} \times F \setminus \{i\}} . $
If $ v, w $ both belong to $\{ 0, 1\}^{ \{ - \ell , \ldots , -1\} \times F  \setminus \{i\}} $  we write 
$$ 
v_{ \{j \}^c }   = w_{ \{ j \}^c }     \mbox{ if  and only if }  v_{-\ell }^{-1}\big(F\setminus\{j\}\big)=
w_{-\ell }^{-1}\big(F\setminus\{j\}\big)  .
$$
In words, the equality $v_{ \{j \}^c }   = w_{ \{ j \}^c }$ means that  $v$ and $w$ coincide on all but the $j$-th coordinate.

Finally, for each $w\in\mathcal{T}_{(i,n)}$ and for any $ j \in F \setminus \{i\} $ we define the set
\begin{equation*}
\mathcal{T}^{w,j}_{(i,n)}= \Big\{v \in \mathcal{T}_{(i,n)}: |v|=|w|,  v_{ \{j \}^c }   = w_{ \{ j \}^c }  \Big\}
\end{equation*}
and introduce the \textit{measure of sensibility}
$$
\Delta_{(i,n)}(j)=\max_{w\in \mathcal{T}_{(i,n)}}\max_{v \in \mathcal{T}^{w,j}_{(i,n)}}|\hat{p}_{(i,n)}(1|w)-\hat{p}_{(i,n)}(1|v )|.
$$
Our interaction neighborhood estimator is defined as follows. 

\begin{definition}
For any positive threshold parameter $\epsilon>0$, the estimated interaction neighborhood of neuron $i\in F,$ at accuracy $\epsilon ,$ given the sample $X_1 (F) , \ldots , X_n (F) ,$  is  defined as
\begin{equation}\label{eq:estimator}
\hat{V}^{(\epsilon)}_{(i,n)}=\{j\in F\setminus\{i\}: \Delta_{(i,n)}(j)>\epsilon\}.
\end{equation}
\end{definition}


\subsection{Consistency of the selection procedure for finite and fully observed interaction neighborhoods}
\label{sec:2.3} 
To ensure the existence of our process we impose the 
following conditions. 

\begin{assumption}\label{ass:1}
Suppose that 
$$ r 
= \sup_{i\in I} \sum_{j\in I } |W_{j\to i } |< \infty .
$$ 
\end{assumption}

\begin{assumption}\label{ass:2}
Suppose that  $ g (t) = \sup_{ j \in I } g_j ( t) < \infty $ for all $t \geq 1 .$ 
\end{assumption}

We define the set  $\Omega^{adm} $  of \textit{admissible pasts} as follows  
\begin{equation}
\label{def:adm_set}
\Omega^{adm}=\left\{x  \in \{0,1\}^{I\times \{\ldots,-1,0\}}: \forall \ i\in I, \ \exists \ \ell_i\leq 0 \ \mbox{with} \  x_{\ell_i}(i)=1 \right\} .
\end{equation}
Observe that if $X_{-\infty}^{0}=x\in \Omega^{adm} $, then $L^i_0>-\infty$ for all $i\in I$.
Therefore, Assumptions \ref{ass:1} and \ref{ass:2}  assure that for each $i\in I$,
$$
\sum_{j\in I}W_{j\to i}\sum_{s=L_{0}^i+1}^{0} g_j(1 -s)X_s(j)<\infty, 
$$ 
which implies that the transition probability $\P(X_{1}(i)=1|X^{0}_{-\infty}=x)$ is well-defined. By induction, for each $t\geq 0$, the transition probabilities \eqref{def:2} are also well-fined. Thus, the existence of the stochastic chain $(X_t)_{t \in \Z }$, starting from  $X_{-\infty}^{0}=x\in \Omega^{adm} $, follows immediately.
Observe that  we do not assume stationarity of the chain. To prove the consistency of our estimator we impose also

\begin{assumption}\label{ass:4}
For all $ i \in I, $  $\varphi_i \in C^1 ( \R, [0, 1]  ) $ is a strictly increasing function. Moreover, there exists a $p_{* } \in ]0, 1 [ $ such that for all $i\in I$ and $u\in \R$
$$
p_* \leq \varphi_i(u)\leq 1-p_{*}.
$$  
\end{assumption}
Define for $i\in I$,
\begin{equation}
\label{def:Kil}
K_{i}=\left[\ \sum_{j\in V^{-}_i}W_{j\to i}g_j(1),\sum_{j\in V^{+}_i}W_{j\to i}g_j(1)\right],
\end{equation}
where $V^{+}_i=\{j\in V_i: W_{j\to i}>0\}$ and $V^{-}_i=\{j\in V_i: W_{j\to i}<0\}.$ 

Notice that under Assumptions \ref{ass:1} and \ref{ass:2},  this interval is always bounded. Finally, we define
\begin{equation}\label{eq:mi}
m_i=\inf_{u\in K_{i}}\left\{\varphi'_i(u)\right\}\inf_{j\in V_i}\left\{|W_{j\to i}|g_j(1)\right\}.
\end{equation}

The following theorem is our first main result. It states the strong consistency of the interaction neighborhood estimator when $V_i \subset F  $.  
By strong consistency we mean that the estimated interaction neighborhood of a fixed 
neuron $i$ equals $V_i$ eventually almost surely as $n\to \infty.$ 

\begin{thm}\label{thm:1}
Let $F \subset I$ be a finite set and  $X_1 (F) ,\ldots, X_n(F) $ be a sample produced by a the stochastic chain $(X_t)_{t\in\Z}$ compatible with \eqref{def:1} and \eqref{def:2}, starting from $X^{0}_{-\infty}=x$ for some fixed $x\in\Omega^{adm} .$ Under Assumptions \ref{ass:1}--\ref{ass:4}, for any $i \in F$ such that $V_i  \subset F,$  the following holds.   
\\
1. {\bf (Overestimation).} \label{thm:1I} For any $j\in F \setminus V_i$, we have that for any $\epsilon>0,$
$$
\P\Big(j\in \hat{V}^{(\epsilon)}_{(i,n)}\Big)\leq  4n^{3/2-\xi}\exp\left\{-\frac{\epsilon^2 n^{2\xi}}{2}\right\}.
$$
2. {\bf (Underestimation).} \label{thm:1II} 
The quantity $m_i$ defined in \eqref{eq:mi} satisfies $ m_i > 0 ,$ and  for any $j\in V_i$ and $0<\epsilon<m_i$,
$$
\P\left(j\notin \hat{V}^{(\epsilon)}_{(i,n)}\right)\leq 4\exp\left\{-\frac{(m_{i}-\epsilon)^2 n^{2\xi}}{2}\right\}+\exp\left\{-O\left(n^{1/2+\xi}\right)\right\}.  
$$
3. In particular, if we choose  $ \epsilon_n = O (n^{ - \xi/2} ),$ where $ \xi $ is the parameter appearing in \eqref{eq:xi}, then
$$  \hat{V}^{(\epsilon_n)}_{(i,n)} =V_i  \mbox{ eventually almost surely.}$$
\end{thm}
The proof of Theorem \ref{thm:1} is given in Section \ref{sec:proof.finite.case}.

\subsection{Extension to the case of partially observed interaction neighborhoods}
We now discuss the case when $V_i $ is not fully included in the sampling region $F,$ in particular, the case when $V_i $ is infinite. In this case, we also impose the following assumptions.

\begin{assumption}\label{ass:6}
$\gamma=\sup_{j\in I}  \|\varphi_j' \|_\infty < \infty .$
\end{assumption}

\begin{assumption}\label{ass:5}
There exists a constant $C$ and $p \geq 1 , $ such that $ g (t) = \sup_{j \in I} g_j (t)   \le C ( 1 + t^p ) $ for all $t \geq 1 . $  
\end{assumption}

Let $\ell_\infty = \{ \xi = (\xi_j)_{j \in I} :  \ \forall j\in I, \xi_j \in \R \ \mbox{and} \   \| \xi\|_\infty := \sup_{j\in I} | \xi_j | < \infty \}  $ be the space of all bounded series of real numbers indexed by $I.$ Under Assumption \ref{ass:5}, we may introduce, for each $t\geq 1$, the continuous operator  $ H (t) : \ell_\infty \to \ell_\infty $ defined by $ (H(t) \xi)_j  = \sum_{k\in I} H_{j, k } (t) \xi_k , $ for all $ j \in I,$ where
\begin{equation}\label{eq:H} H_{j, k } (t) := \gamma | W_{k \to j }| g_k  (t ) , \mbox{ for } j \neq i , 
\end{equation}
and, for $p_* $ as in Assumption \ref{ass:4},
$$ H_{i, i } (t) := ( 1 - p_*) 1_{ \{ t = 1 \}}.$$

By our assumptions, the norm of the operator $H (t)$ defined by 
$$ \|\! | H (t) \|\!| = \sup \{   \| H \xi \|_\infty  : \xi \in \ell_\infty , \| \xi \|_\infty = 1 \}$$ 
satisfies
$$ \|\! | H (t) \|\!| \le C \gamma r (1 + t^p) + (1 - p_*) 1_{ \{ t= 1\} }.$$
Then for any $ \alpha > 0 , $ the linear operator  
$$ \Lambda  (\alpha ) = \sum_{ t=1}^\infty  e^{ - \alpha t } H (t) $$
is well-defined and continuous as well. In particular,  there exists $\alpha_0 \geq 0$ such that 
\begin{equation}\label{eq:alpha_0}
\| \!| \Lambda ( \alpha_0) \|\!| < 1 .
\end{equation}

We are now ready to state our second main result. It gives precise error bounds for the interaction neighborhood estimator when $V_i $ is not fully observed. These error bounds depend on the tail of the series 
\begin{equation}\label{eq:tail}
\Sigma_i (F) := \sum_{ j  \notin V_i \cap F} | W_{ j \to i } | .
\end{equation}
To state the theorem we shall also need the definitions
\begin{equation*}
K^{[F]}_{i}=\left[\ \sum_{j\in V^{-}_i \cap F}W_{j\to i}g_j(1),\sum_{j\in V^{+}_i \cap F}W_{j\to i}g_j(1)\right] 
\end{equation*}
and 
$$ m^{[F]}_i=\inf_{u\in K^{[F]}_{i}}\left\{\varphi'_i(u)\right\}\inf_{j\in V_i \cap F}\left\{|W_{j\to i}|g_j(1)\right\}.$$

\begin{thm}\label{thm:main2}

Let $F \subset I$ be a finite set and  $X_1 (F) ,\ldots, X_n(F) $ be a sample produced by a the stochastic chain $(X_t)_{t\in\Z}$ compatible with \eqref{def:1} and \eqref{def:2}, starting from $X^{0}_{-\infty}=x$ for some fixed $x\in\Omega^{adm} .$ Under Assumptions \ref{ass:1}--\ref{ass:5}, for any $i \in F$ such that $ V_i \cap F \neq \emptyset,$ the following assertions hold true. 
\\
1. {\bf (Overestimation).} \label{thm:1I} For any $j \in F \setminus V_i$, we have that for any $\epsilon>0,$ 
$$
\P\Big(j\in \hat{V}^{(\epsilon)}_{(i,n)}\Big)\leq  4n^{3/2-\xi}\exp\left\{-\frac{\epsilon^2 n^{2\xi}}{2}\right\} + C  (e^{ \alpha_0  n } \vee n ) \Sigma_i (F).
$$
 \\
2. {\bf (Underestimation).} \label{thm:1II} 
We have that $m^{[F]}_i>0 ,$ and for any $j\in V_i \cap F $ and $0<\epsilon<m^{[F]}_i$,
\begin{multline*}
\P\left(j\notin \hat{V}^{(\epsilon)}_{(i,n)}\right)\leq 4\exp\left\{-\frac{(m^{[F]}_{i}-\epsilon)^2 n^{2\xi}}{2}\right\}+\exp\left\{-O\left(n^{1/2+\xi}\right)\right\} \\
+ C  (e^{ \alpha_0  n } \vee n ) \Sigma_i (F) .
  \end{multline*}
\end{thm}

The proof of Theorem \ref{thm:main2} is given in Section \ref{sec:proofs.infinite.case}.

\section{Exponential inequalities}
\label{sec:dev.ineq}
To prove Theorems \ref{thm:1} ans \ref{thm:main2} we need some exponential inequalities, including a new conditional Hoeffding-type inequality, stated in Proposition \ref{prop:1} below, which is interesting by itself. 

For each $\ell\geq 1,$ $F\subset I$ finite and $w\in \{0, 1 \}^{ \{ - \ell, \ldots, - 1 \} \times F \setminus \{i\} }$, we write
\begin{equation}\label{def:transi_prob_F}
p_i ( 1 | w ) =\P(X_0(i)=1|X^{-1}_{-\ell-1}(i)=10^{\ell},X^{-1}_{-\ell}(F \setminus \{i\} )=w).
\end{equation} 
By homogeneity of  the transition probability \eqref{def:2}, this implies that for any $t\in \Z$, 
$$
p_i ( 1 | w )=\P(X_t(i)=1|X^{t-1}_{t-\ell-1}(i)=10^{\ell},X^{t-1}_{t-\ell}(F \setminus \{i\})=w).
$$
Moreover, we also have that $p_i ( 1 | w )=p_i(1|w(V_i))$ for any set $F\supset V_i$, where $w(V_i)$ is the configuration $w$ restricted to the set $V_i$.

\begin{proposition}
\label{prop:1}
Suppose that $V_i$ is finite and $ V_i \subset F .$ Then for any $\ell\geq 1$, $w\in \{0, 1 \}^{ \{ - \ell, \ldots, - 1 \} \times F\setminus\{i\} } $, $\lambda>0$ and all $t> \ell +1$, 
\begin{equation}
\P(|M_{(i,t)}(w)|> \lambda)\leq 2\exp\left\{-\frac{2 \lambda ^2}{t-\ell +1}\right\}\P(N_{(i,t)}(w)>0),
\end{equation}
where $M_{(i,t)}(w):=N_{(i,t)}(w,1)-p_i(1|w)N_{(i,t)}(w)$.  
\end{proposition}
\begin{proof}
We denote $p=p_i(1|w)$ and for each $t\geq \ell+1,$ $N_{(i,t)}(w)=N_t$, $Y_t=\one\{X_t(i)=1\}-p$, $\chi_t=\one\{X_{t-\ell}^{t-1}(F\setminus\{i\})=w, X^{t-1}_{t-\ell-1}(i)=10^{\ell}\}$ and also $M_{(i,t)}(w)=M_t$ with the convention that $M_{\ell+1}=0.$ Thus for $t\geq \ell+2$,
\begin{equation}
\label{decomp.}
M_t=M_{t-1}+\chi_t Y_t.
\end{equation}
Since $\P(M_t> \lambda)=\P(M_t> \lambda, N_{t}>0)$, the Markov inequality implies that 
$$
\P(M_t> \lambda)\leq e^{-\lambda\sigma}E\big[e^{\sigma M_t}\one\{N_t>0\}\big],
$$
for all $\sigma>0$. Notice that $\{N_t>0\}=\{N_{t-1}>0\}\cup \{N_{t-1}=0,\chi_t=1\}$, so that by \eqref{decomp.}, it follows that $\E\big[e^{\sigma M_t}\one\{N_t>0\}\big]$ can be rewritten as 
\begin{equation}
\label{rewriting_Mn}
\E\big[e^{\sigma M_{t-1}}e^{\sigma\chi_t Y_t}\one\{N_{t-1}>0\}\big]+
\E\big[e^{\sigma Y_t}\one\{N_{t-1}=0,\chi_t=1\}\big].
\end{equation}
From the assumption $V_i\subset F$ it follows that $p=p_i(1|w)=p_i(1|w(V_i))$ and $\E\big[\chi_t Y_t\big|\cF_{t-1}]=0$. Since $- p \le \chi_t Y_t \leq 1 - p$, the classical Hoeffding bound implies that $\E\big[e^{\sigma\chi_t Y_t}\big|\cF_{t-1}]\leq e^{\sigma^2/8}$ and therefore the expression \eqref{rewriting_Mn}  can be bounded above by
$$
\E\big[e^{\sigma M_t}\one_{\{N_t>0\}}\big]\leq e^{\sigma^2/8}\E\big[e^{\sigma M_{t-1}}\one_{\{N_{t-1}>0\}}\big]+
e^{\sigma^2/8}\E\big[\one_{\{N_{t-1}=0,\chi_t=1\}}\big].
$$  
By iterating the inequality above and using the identity 
$$
\one\{N_{t}>0\}=\one\{N_{\ell+1}>0\}+ \sum_{s=\ell+2}^t\one\{N_{s-1}=0,\chi_s=1\},
$$
we obtain that $\E\big[e^{\sigma M_t}\one\{N_t>0\}\big]\leq e^{(t-\ell+1)\sigma^2/8}\P(N_t>0)$. 
Thus, collecting all these estimates, we deduce, by taking $\sigma=4\lambda (t-\ell+1)^{-1}$, that 
\begin{equation*}
\P(M_t> \lambda)\leq \exp\left\{-\frac{2\lambda ^2}{t-\ell +1}\right\}\P(N_t>0).
\end{equation*}
The left-tail probability $\P(M_t< -\lambda)$ is treated likewise. 
\end{proof}

As a consequence of Proposition \ref{prop:1}, we have the following result.
\begin{proposition}
\label{prop:2}
Suppose that $V_i$ is finite and $ V_i \subset F .$ Then for any $\ell\geq 1$, $t> \ell + 1$, $w\in\{0, 1 \}^{ \{ - \ell, \ldots, - 1 \} \times F\setminus\{i\} } $,  $\xi\in(0,1/2)$ and $\epsilon>0$, we have
\begin{multline*}
\P\Big(|\hat{p}_{(i,t)}(1|w)-p_i(1|w)|>\epsilon,N_{(i,t)}(w)\geq t^{1/2+\xi}\Big) \\
\leq 2\exp\left\{-2\epsilon^2 t^{2\xi}\right\} \P (N_{(i,t)}(w) > 0 ).
\end{multline*}
\end{proposition}
\begin{proof}
For any $w\in\{0, 1 \}^{ \{ - \ell, \ldots, - 1 \} \times F\setminus\{i\} }  $ and $t>\ell + 1$, 
$$
M_t(w)=\big(\hat{p}_{(i,t)}(1|w)-p_i(1|w)\big)N_{(i,t)}(w),
$$ 
so that 
$$
\P\Big(|\hat{p}_{(i,t)}(1|w)-p_i(1|w)|>\epsilon,N_{(i,t)}(w)\geq t^{1/2+\xi}\Big)\leq \P(|M_t(w)|>\epsilon t^{1/2+\xi}).
$$
Thus the result follows from Proposition \ref{prop:1} by taking $\lambda =\epsilon t^{1/2+\xi}.$
\end{proof}

The next two results will be used to control the probability of underestimating $V_i .$ We start with a simple lower bound 
which follows immediately from Assumption \ref{ass:4}.
\begin{lemma}\label{lem:2}
For any fixed $i\in F$, $t > \ell+1 $ and  $w\in \{0, 1 \}^{\{-\ell, \ldots, - 1 \} \times F\setminus \{i\} } $, we define 
for $1 \leq s \leq t-\ell,$ 
$$
Z_s=\one\{X_{s}^{s+\ell}(i)=10^{\ell},
X_{s+1}^{s+\ell}(F\setminus\{i\})=w\}.
$$
Under Assumption \ref{ass:4}, it follows that
$$
\P\left(Z_s=1|\cF_{s-1}\right)\geq p_{min}^{|F|\ell+1},
$$
where $p_{min}=\min\{p_*,(1-p_*)\}>0$ with $p_*$ as in Assumption \ref{ass:4}. 
\end{lemma}

\begin{lemma}\label{lem:N_n_notsmal}
Suppose Assumption \ref{ass:4}. For any $\xi\in(0,1/2)$, $i\in F$ and $w\in \{0, 1 \}^{\{ - 1 \} \times F\setminus \{i\} } ,$ it holds that
$$
\P\left(N_{(i,t)}(w)<t^{1/2+\xi}\right)\leq \exp\left\{-O\left(t^{1/2+\xi}\right)\right\}.
$$
\end{lemma}
\begin{proof}
For each $1 \leq s \leq t- 1,$ let $Z_s$ be the random variable defined as in Lemma \ref{lem:2} with $\ell = 1.$ Now we define $Y_s=Z_{ 2 (s-1) +1}$ for  $1 \leq s \leq \floor{t/2}$ and observe that $\mathcal{G}_s:=\sigma(Y_1,\ldots, Y_{s})\subset \cF_{2 s}.$ Thus, by Lemma \ref{lem:2}, 
$$
\P(Y_s=1|\mathcal{G}_{s-1})=\E\Big[\P\Big(Z_{2 (s-1)+1}=1|\cF_{2 (s-1)}\Big)|\mathcal{G}_{s-1}\Big]\geq p_{min}^{|F| +1}. 
$$
Define $q_*= p_{min}^{|F| +1}$. Then Lemma A.3 of \cite{CT} implies for every $ \nu \in ] 0, 1 [ , $ 
\begin{equation*}
\P\left(\frac{1}{\floor{t/2}}\sum_{s=1}^{\floor{t/2}}Y_s< \nu q_*  \right) \leq \exp\left\{- \floor{t/2}\frac{q_*}{4}\Big(1-\nu \Big)^2\right\}.
\end{equation*}
Clearly $N_{(i,t)}(w)=\sum_{s=1}^{t-1}Z_s\geq \sum_{s=1}^{\floor{t/2}}Y_s$, so that it follows from the inequality above that 
$$
\P\left(N_{(i,t)}(w)< \nu q_*  \floor{t/2} \right) \leq \exp\left\{- \floor{t/2}\frac{q_*}{4}\Big(1-\nu \Big)^2\right\}.
$$
Finally, for any fixed $\nu\in (0,1)$ and all $t$ large enough, $\floor{t/2}\frac{q_*}{4}(1-\nu )^2> t^{1/2+\xi}$ and $\nu q_*  \floor{t/2} > t^{1/2+\xi}$, implying the assertion. 

\end{proof}

\section{Proof of Theorem \ref{thm:1}}\label{sec:proof.finite.case}
Suppose that $V_i \subset F  $ and notice that for any $\ell\geq 1$ and $w\in \{0, 1 \}^{ \{ - \ell, \ldots, - 1 \} \times V_i}$, it holds that
\begin{equation}
\label{def:transition_proba}
p_i ( 1 | w )=\varphi_i \left( \sum_{j\in V_i} W_{j \to i } \sum_{ t = - \ell}^{ - 1} g_j ( -t) w_t (j) \right). 
\end{equation}

\begin{proof}[Proof of Item 1 of Theorem \ref{thm:1}]
Using the definition of $\hat{V}^{( \epsilon)}_{(i,n)}$ and applying the union bound, we deduce that
\begin{eqnarray}
\label{prop3:ineq.1}
\P\big(j\in \hat{V}^{(\epsilon)}_{(i,n)}\big)&= & \P(\Delta_{(i,n)}(j)>\epsilon)\nonumber \\
&\leq & \E\left[\sum_{(w,v) }
\one\left\{A^{w,v,j}_{(i,n)},|\hat{p}_{(i,n)}(1|w)-\hat{p}_{(i,n)}(1|v)|>\epsilon \right\}\right],
\end{eqnarray}
where $A^{w,v,j}_{(i,n)}:=\{(w,v)\in \mathcal{T}_{(i,n)}\times \mathcal{T}^{w,j}_{(i,n)}\}$.
Since $j\notin V_i$ and $V_i\subset F$, the configurations of any pair 
$(w,v)\in \mathcal{T}_{(i,n)}\times\mathcal{T}^{w,j}_{(i,n)} $ coincide in restriction to the set $V_i.$ In other words,  $w(V_i)= v(V_i) .$ In particular, it follows from \eqref{def:transition_proba} that $p_i(1|w)=p_i(1|w(V_i))=p_i(1|v(V_i))=p_i(1|v)$. 

Therefore, applying the triangle inequality, it follows that on $A^{w,v,j}_{(i,n)}$,
\begin{equation*}
\one{\left\{|\hat{p}_{(i,n)}(1|w)-\hat{p}_{(i,n)}(1|v)|>\epsilon\right\}}\leq \sum_{u\in\{w,v\}} 
\one{\left\{|\hat{p}_{(i,n)}(1|u)-p_{i}(1|u)|>\epsilon /2\right\}},
\end{equation*}
so that the expectation in \eqref{prop3:ineq.1} can be bounded above by
\begin{equation}
\label{prop3:ineq.2}
2\E\left[\sum_{w }
\one_{\left\{w\in\mathcal{T}_{(i,n)}\right\}}
\sum_{v }\one_{\left\{
v\in\mathcal{T}_{(i,n)}\right\}}\one\left\{|\hat{p}_{(i,n)}(1|v)-p_{i}(1|v)|>\epsilon  /2\right\}\right].
\end{equation}
Now, since $\sum_{w}N_{(i,n)}(w) \le n ,$ we have that
$$
n \geq \sum_{w :\ N_{(i,n)}(w)\geq n^{1/2+\xi}}N_{(i,n)}(w) \, \geq \, n^{1/2+ \xi} \; | \{  w :\ N_{(i,n)}(w)\geq n^{1/2+\xi}\}  | ,
$$
which implies that 
$|\mathcal{T}_{(i,n)}|\leq n^{1/2-\xi}.$ From this last inequality and Proposition \ref{prop:2}, which is stated in Section \ref{sec:dev.ineq} below, we obtain the following upper bound for \eqref{prop3:ineq.2},
\begin{equation}
\label{prop3:ineq.3}
4n^{1/2-\xi}\exp\left\{-\frac{\epsilon^2 n^{2\xi}}{2}\right\}\E\left[\sum_{w }
\one{\left\{N_{(i,n)}(w)>0\right\}}
\right].
\end{equation}  
Since $\sum_{w}
\one{\left\{N_{(i,n)}(w)>0\right\}}\leq n$, the result follows from  inequalities \eqref{prop3:ineq.1} and \eqref{prop3:ineq.3}.
\end{proof}

Before proving Item 2 of Theorem \ref{thm:1},  we will prove the following lemma.

\begin{lemma}
\label{lem:4}
Suppose that $V_i$ is finite and define for each $j\in V_i,$  
$$
m_{i,j}:=\max_{w,v \in \{0, 1 \}^{\{- 1 \} \times V_i }:w_{\{j\}^c}=v_{\{j\}^c} }|p_i(1|w)-p_i(1|v)| .
$$
Then, under Assumption \ref{ass:4}, we have that
\begin{equation}
\label{Keybound}
\inf_{j\in V_i}m_{i,j}\geq \inf_{x\in K_{i}} \left\{\varphi_i'(x)\right\} \inf_{j\in V_i}\left\{|W_{j\to i}|g_j(1)\right\}=m_{i}>0,
\end{equation}
where $K_{i}$ is defined in \eqref{def:Kil}.
\end{lemma}
\begin{proof}
For each $j\in V_i$ take any pair $w,v \in \{0, 1 \}^{\{- 1 \} \times V_i }$  such that $w_{\{j\}^c}=v_{\{j\}^c}$ with 
$w_{-1}(j)=1$ and $v_{-1}(j)=0$.  
By Assumption \ref{ass:4}, the function $\varphi_i$ is differentiable such that, for $\ell =  1 , $ 
$$
|p_i(1|w)-p_i(1|v)|\geq \inf_{x\in K_{i}}\left\{\varphi_i'(x)\right\}\Big|\sum_{k\in V_i}W_{k\to i}\sum_{t= \ell  }^{-1}g_k(-t)(w_{t}(k)-v_{t}(k))\Big|.
$$
Since $|\sum_{k\in V_i}W_{k\to i}\sum_{t=-\ell}^{-1}g_k(-t)(w_{t}(k)-v_{t}(k))\Big|=|W_{j\to i}|g_j(1)$, the inequality above implies the first assertion of the lemma.

By Assumption \ref{ass:4}, the function $\varphi_i$ is strictly increasing ensuring that $\inf_{x\in K_{i}}\left\{\varphi_i'(x)\right\}>0$. Thus, since for all $j\in I$ the sequence $g_j$ is strictly positive and $V_i \neq \emptyset $ is finite, we clearly have that $m_{i}>0$.    
\end{proof}

We are now in position to conclude the proof of Theorem \ref{thm:1}.

\begin{proof}[Proof of Item 2 of Theorem \ref{thm:1}]
Lemma \ref{lem:4} implies that $m_i$ defined in \eqref{Keybound} is positive. Let $0<\epsilon<m_{i}$. If $j\in V_i$, Lemma \ref{lem:4} implies the existence of strings $w^*,v^*\in\{0,1\}^{\{-1\}\times F\setminus\{i\}}$ such that $w^*_{\{j\}^c}=v^*_{\{j\}^c}$ and 
$$
|p_i(1|w^*)-p_i(1|v^*)|=|p_i(1|w^*(V_i))-p_i(1|v^*(V_i))|  \geq m_{i}.
$$ 
Denoting by $C_n=\{N_{(i,n)}(w^*)\geq n^{\xi+1/2},N_{(i,n)}(v^*)\geq n^{\xi+1/2}\}$ it follows that
\begin{equation}
\label{prop.3:ineq.1}
\P\left(j\notin \hat{V}_{(i, n)}^{(\epsilon)}\right)\leq \P(|\hat{p}_{(i,n)}(1|w^*)-\hat{p}_{(i,n)}(1|v^*)|<\epsilon, C_n ) +\P(C^c_n).
\end{equation} 
Now notice that the first term on the right in \eqref{prop.3:ineq.1} is upper bounded by
$$
\sum_{u\in\{w^*,v^*\}}\P(|\hat{p}_{(i,n)}(1|u)-p_i(1|u)|>(m_{i}-\epsilon)/2, N_{(i,n)}(u)\geq n^{\xi+1/2} ),
$$
and since $m_{i}>\epsilon$, the result follows from Proposition  \ref{prop:2} and Lemma \ref{lem:N_n_notsmal}, both stated in Section \ref{sec:dev.ineq} above.
 
\end{proof}

\begin{proof}[Proof of Item 3 of Theorem \ref{thm:1}]
Define for $n\in\N$ the sets 
$$
O_n=\left\{j\in F\setminus V_i: j\in V_{(i,n)}^{(\epsilon_n)}\right\} \ \mbox{and} \ U_n=\left\{j\in V_i: j\notin V_{(i,n)}^{(\epsilon_n)}\right\}.
$$ 
Applying the union bound and then Item 1, we infer that 
$$
\P(O_n)\leq 4\left(|F|-|V_i|\right)n^{3/2-\xi}\exp\left\{-\frac{\epsilon_n^2 n^{2\xi}}{2}\right\}.
$$
Applying once more the union bound and then using Item 2, we also infer that 
$$
\P(U_n)\leq |V_i|\left(4\exp\left\{-\frac{(m_{i}-\epsilon_n)^2 n^{2\xi}}{2}\right\}+\exp\left\{-O\left(n^{1/2+\xi}\right)\right\}\right).
$$
Since $\{V^{(\epsilon)}_{(i,n)}\neq V_i\}=O_n\cup U_n$, we deduce that
$
\sum_{n=1}^{\infty}\P\left(V^{(\epsilon_n)}_{(i,n)}\neq V_i\right)< \infty,
$
so that the result follows from the Borel-Cantelli Lemma.
\end{proof}

\section{Proof of Theorem \ref{thm:main2}}
\label{sec:proofs.infinite.case}
To deal with the case $V_i \not \subset F ,$ we couple the process $X=(X_t  )_{t \in \Z} $ with its fixed range approximation $ X^{[F]}=(X^{[F]}_t  )_{t \in \Z}  , $ where $ X^{[F]} $ follows the same dynamics as $X$, defined in \eqref{def:1} and \eqref{def:2} for all $ j \neq i,$ except that \eqref{def:2} is replaced -- for the fixed neuron $i$ -- by 
\begin{equation}\label{def:2bis}
\P(X^{[F]}_{t+1}(i)=1|\cF_{t})= \varphi_i\Big( \sum_{j\in V_i \cap F  }W_{j\to i}\sum_{s=L_{t}^{i,[F]}+1}^{t} g_j(t+1 -s)X^{[F]}_s(j)\Big).
\end{equation}
Moreover, we suppose that $X$ and $X^{[F]} $ start from the same initial configuration $X_{-\infty}^0 =(X^{[F]})_{-\infty}^0 = x , $ where $x \in \Omega^{adm}.$ 

We will show in  Proposition \ref{thm:couplage} in the Appendix that
Assumptions  \ref{ass:1}, \ref{ass:6} and \ref{ass:5}  imply the existence of a coupling between $X$ and $X^{[F]}$ and of a constant $ C > 0 $  such that 
\begin{equation}\label{eq:discrepancy} \sup_{j\in I }  \P  \left(  \exists t \in [1, n ] : X_t (j ) \neq X_t^{[F]} (j) \right)  \le C ( e^{ \alpha_0  n } \vee n) \sum_{ j \notin V_i \cap F} | W_{j \to i } | . \end{equation}

Write 
$$ 
E_n = \bigcap_{ 1 \le s \le n  } \{ X_{s } (i ) = X_s^{ [ F ] } (i )   \} .
$$
On $E_n, $ instead of working with $X_s (i) , 1 \le s \le n , $ we can therefore work with its approximation $X_s
^{[F]} ( i ) , 1 \le s \le n ,$ having conditional transition probabilities (for neuron $i$) given by
$$ p_i^{[F]} ( 1 | w) = \varphi_i \Big( \sum_{j\in V_i \cap F } W_{j \to i } \sum_{ s = - \ell}^{ - 1} g_j ( -s) w_j (s) \Big)  $$ 
which only depend on $ w (V_i\cap F) .$ As a consequence, on $E_n $ the proof of Theorem \ref{thm:main2} works as in the preceding section, except that we replace $m_i$ by
\begin{equation*}
m^{[F]}_{i}=\inf_{x\in K^{[F]}_{i}} \left\{\varphi_i'(x)\right\} \inf_{j\in V_i \cap F }\left\{|W_{j\to i}|g_j(1)\right\}>0 
\end{equation*}
if $ V_i \cap F \neq \emptyset.$
Here $K^{[F]}_{i}$ is defined by
\begin{equation*}
K^{[F]}_{i}=\Big[\ \sum_{j\in V^{-}_i \cap F}W_{j\to i}g_j(1),\sum_{j\in V^{+}_i \cap F}W_{j\to i}g_j(1)\Big] .
\end{equation*}
Finally, writing 
$$
O_n=\left\{j\in F\setminus V_i: j\in V_{(i,n)}^{(\epsilon_n)}\right\} \ \mbox{and} \ U_n=\left\{j\in V_i \cap F : j\notin V_{(i,n)}^{(\epsilon_n)}\right\},
$$
we obtain 
$$ \P (O_n) \le \P ( O_n \cap E_n ) + \P ( E_n^c) , \; \P (U_n) \le \P ( U_n \cap E_n ) + \P ( E_n^c),$$
where as before in Theorem \ref{thm:1},
$$ \P(O_n \cap E_n )\leq 4 n^{3/2-\xi}\exp\left\{-\frac{\epsilon_n^2 n^{2\xi}}{2}\right\} 
$$
and 
$$
\P(U_n \cap E_n ) )\leq |F |\left(4\exp\left\{-\frac{(m^{[F]}_{i}-\epsilon_n)^2 n^{2\xi}}{2}\right\}+\exp\left\{-O\left(n^{1/2+\xi}\right)\right\}\right).$$
Finally, by inequality \eqref{eq:discrepancy},
$$ \P( E_n^c) =   \P ( \exists t \in [1, n ] :  X_t (i) \neq X_t^{[F]} (i)  ) \le C  ( e^{ \alpha_0n } \vee n) \sum_{ j \notin V_i \cap F} | W_{j \to i } |  ,$$
for some constant $C.$ This concludes the proof.

\section{Time complexity of the estimation procedure}
\label{sec:proof.runtime}
The time complexity of our selection procedure has quadratic growth with respect to $n,$ the length of the time interval during which the neural network is observed.  This is the content of the following proposition. 
\begin{proposition}
\label{thm:runtime}
The number of operations needed to compute the set $\hat{V}^{(\epsilon)}_{(i,n)}$ is $O(n^2)$.
\end{proposition}
\begin{proof}

All the random variables involved in the definition of the random set $\hat{V}^{(\epsilon)}_{(i,n)}$ can be written in terms of the counting variables ${N}_{(i,n)}(\cdot)$. All  counting variables ${N}_{(i,n)}(w)$ for $w\in \{0,1\}^{\{-\ell,\ldots, -1\}\times F\setminus\{i\}}$ with fixed length $\ell$ can be computed simultaneously  after $n-\ell-1$ operations. Indeed, we set initially $N_{(i,n)}(w)=0$ for all pasts $w\in \{0,1\}^{\{-\ell,\ldots, -1\}\times F\setminus\{i\}}$ and then we increment by $1$ the count of the past that has occurred at time $\ell+2\leq t\leq n$, leaving the counts of all other pasts unchanged.
Thus with 
$$
\sum_{l=1}^{\floor{n-2-n^{1/2+\xi}}}(n-\ell-1)\leq O\big(n^2-n^{1+2\xi}\big)
$$
operations we compute all the counting variables
${N}_{(i,n)}(w)$, for all local pasts $w\in \{0,1\}^{\{-\ell,\ldots, -1\}\times F\setminus\{i\}},$ for all $1\leq \ell\leq \floor{n-2-n^{1/2+\xi}}$, where for each $x\in \R$, $\floor{x}$ is the largest integer less than or equal to x $ $.

Now, given all counting variables, to compute $\Delta_{(i,n)}(j)$ we need at most $O\big(\sum_{w\in \mathcal{T}_{(i,n)}}|\mathcal{T}^{w,j}_{(i,n)}|\big)\leq O(n^{1-2\xi})$ computations which in turns implies that, given all counting variables, with at most $O(|F|n^{1-2\xi})$ operations we compute our estimator $\hat{V}^{(\epsilon)}_{(i,n)}$. Therefore, in the overall, we need to perform at most 
$$
O\big(n^2-n^{1+2\xi}\big)+O(|F|n^{1-2\xi})=O(n^2)
$$   
operations. 
\end{proof} 
     
\appendix

\section{Auxiliary results}
\label{sec:proba}
 In this section,  
we prove the coupling result \eqref{eq:discrepancy} needed in the proof of Theorem \ref{thm:main2}. For that sake, 
let $F\subset I$ be a finite set, fix $i\in F$ and let $ U_t (j ) , j \in I, t \geq 1  , $ be an i.i.d.\ family of random variables uniformly distributed on $ [0, 1 ].$ 

The coupling is defined as follows.
For any $x\in\Omega^{adm}$, we define  $X_t(j)=X^{[F]}_t(j)=x_t(j)$ for each $t\leq 0$ and $j\in I$. For each $t\geq 1$ and $j\in I$, we define

$$
X_t(j)=
\begin{cases}
1, \ \mbox{if} \  U_t(j)>\varphi_j (\eta_{t-1} (j) )\\
0, \ \mbox{if} \  U_t(j)\leq \varphi_j (\eta_{t-1}(j))
\end{cases}
$$
and
$$ 
X^{[F]}_t(j)=
\begin{cases}
1, \ \mbox{if} \  U_t(j)>\varphi_j (\eta^{[F]}_{t-1} (j) )\\
0, \ \mbox{if} \  U_t(j)\leq \varphi_j (\eta^{[F]}_{t-1}(j) ),
\end{cases}
$$
where for each $t\geq 0$ and $j\in I$,
$$
\eta_{t}(j)  = \sum_{k\in V_j}W_{k\to j}\sum_{s=L^j_{t} +1 }^{t} g_k(t+1-s) X_s(k)
$$ 
and,  if $j\neq i$,
\begin{equation}\label{eq:uL}
 \eta^{[F]}_{t} (j)  = \sum_{k\in V_j}W_{k\to j}\sum_{s=L^{j, [F]}_{t} +1 }^{t} g_k(t+1-s)X^{[F]} _s(k),
\end{equation}
and finally
\begin{equation}\label{eq:uLj}
 \eta^{[F]}_{t} (i) = \sum_{k\in V_i \cap F }W_{k\to i}\sum_{s=L^{i, [F]} _{t} +1 }^{t} g_k(t+1-s)X^{[F]} _s(k).
\end{equation}
In other words, the process $X^{[F]}$ has exactly the same dynamics as the original process $X$, except that neuron $i$ depends only on neurons belonging to $ V_i \cap F . $ 
Notice that we use the same uniform random variables $U_t (j) $ to update the values of $ X_t (j) $ and of $ X_t^{[F]} ( j) .$ In this way we achieve a coupling between the two processes. We shall write $\E_x$ to denote the expectation with respect to this coupling. Then we have the following result.

\begin{proposition}\label{thm:couplage}
Assume Assumptions  \ref{ass:1}, \ref{ass:6} and \ref{ass:5}, and   let $ \alpha_0 $ be defined as in \eqref{eq:alpha_0}.
\\
1. If $ \alpha_0 > 0 ,$  then 
\begin{equation}\label{eq:discrepancy3}
\sup_{j\in I }  \P_{x} ( \bigcup_{s=1}^t \{X_s (j ) \neq X_s^{[F]} (j)\}   ) \le C e^{ \alpha_0  t } \sum_{ k \notin V_i \cap F } | W_{k \to i } | . 
\end{equation}
2. Suppose now that $ \alpha_0 = 0 $ and write for any $j\in I$, $\varrho_j  = \sum_{t=1}^\infty g_j (t) ,$ $\varrho = \sup_{j\in I} \varrho_i.$ Then 
\begin{equation}\label{eq:dobrushin}
\chi=(1- p_*) + \gamma  \sup_{j\in I} \sum_{k\in I} \varrho_k | W_{k \to I} |  < 1 ,
\end{equation}
and in this case
\begin{equation}\label{eq:discrepancy4}
\sup_{j\in I }  \P_{x} \Big( \bigcup_{s=1}^t \{X_s (j ) \neq X_s^{[F]} (j)\}   \Big)  \le \frac{ \gamma  \varrho t  }{ 1 -  \chi} \sum_{ k \notin V_i\cap F } | W_{k \to i } |  . 
\end{equation}
\end{proposition}

\begin{proof}
For notational convenience, we assume that the starting configuration $x \in \Omega^{adm}$ satisfies $x_0 (i) =1$ and extend the definition of $g_j$ by defining $g_j(t)=0$ for all t$\leq 0$ and $j\in I$.

We start proving Item 1. Recall  the definition of the continuous operator $ H (t)   $ in \eqref{eq:H}. In the sequel, we set also $H(0)\equiv 0$.

Let for each $t\geq 0$, 
$$ D_j(t ) =  1\{ L^j_t\neq L^{j,[F]}_t\} , j \in I,$$ 
and observe that 
\begin{equation}
\label{Eq:A6}
P_x(X_t(j)\neq X^{[F]}_t(j))\leq E_x[D_j(t)].
\end{equation}
Given $\cF_t$, we update $ D_j(t) $ as follows. If neuron $j$ spikes at time $t+1$ in both processes,  then $D_j (t+1) = 0 $ regardless the value of $D_j(t)$. By the definition of the coupling, this event occurs with probability $ \varphi_j ( \eta_t (j ) \wedge \eta_t^{[F]} ( j) ) \geq p_*  .$ 
When $D_j (t)=1$, then  ${D_j(t+1)=1}$ if and only if neuron $j$ does not spike in both processes. Clearly, this event has probability $1-\varphi_j(\eta_j(t)\wedge \eta^{[F]}_j(t))$. Finally, if $D_j(t)=0$, then ${D_j(t+1)=1}$  if and only if  neuron $j$ spikes only in one of the two processes. This event occurs with probability $| \varphi_j ( \eta_t (j ) ) - \varphi_j (\eta_t^{[F]} (j) ) |$. Thus for all $j\in I$, we have
\begin{multline}\label{eq:ouf}
\E_x ( D_j (t +1 ) | {\cal F}_t) =  D_j ( t) ( 1 - \varphi_j ( \eta_t (j ) \wedge \eta_t^{[F]} (j) ) \\
+   | \varphi_j ( \eta_t (j ) ) - \varphi_j (\eta_t^{[F]} ( j) ) |( 1 - D_j ( t))  .
\end{multline}
Since $\varphi_i $ is Lipschitz with Lipschitz constant $ \gamma $ and $ L_t^i =  L_t^{i, [F]} $ on $\{ D_i ( t) = 0 \} , $ we have on this event, 
\begin{multline}\label{eq:ouff}
\frac{1}{\gamma } \, | \varphi_i ( \eta_t (i ) ) - \varphi_i (\eta_t^{[F]} ( i) ) | \le   | \eta_t (i )  - \eta_t^{[F]} ( i)  | \\
\le \sum_{ k \in V_i \cap F} |W_{k \to i } |\sum_{ s= L_t^i  +1}^t g_k (t+1- s) | X_s (k ) - X_s^{[F]} (k)|  \\
+  \sum_{ k \notin V_i \cap F  } |W_{k \to i } | \sum_{ s=1  }^t  g_k (t+1-s ) \\
\le  \sum_{ k \in V_i \cap F} |W_{k \to i } |\sum_{ s= 1 }^{t+1} g_k  (t+1- s ) | X_s (k ) - X_s^{[F]} (k)|  \\
+ \sum_{ k \notin V_i \cap F  } |W_{k \to i } | \sum_{ s=   1  }^{t+1}  g_k (t+1-s ) ,
\end{multline}
where we have used that $ g_k (0) = 0 $  in order to replace the sum $ \sum_{s=1}^t $ by $ \sum_{s=1}^{t+1}.$ Moreover, we have used that $  L_t^i =  L_t^{i, [F]} \geq 0 $ for all $ t \geq 0, $ by our choice of $ x.$ 

Similarly, for all $j\neq i$, we have on $\{ D_j ( t) = 0 \} , $
\begin{equation}
\label{eq:ouff1}
\frac{1}{\gamma } \, | \varphi_j ( \eta_t (j ) ) - \varphi_j (\eta_t^{[F]} ( j) ) | \le \sum_{ k \in V_j } |W_{k \to j } |\sum_{ s= 1 }^{t+1} g_k  (t+1- s ) | X_s (k ) - X_s^{[F]} (k)|.
\end{equation}

For each $j\in I$, let $ \delta_j ( t) = \E_x ( D_j (t) ) $ and write $ \delta (t) = (\delta_j (t) )_{j \in I}$ for the associated column vector. Taking expectation in \eqref{eq:ouf}--\eqref{eq:ouff1} and using that $ \E_x | X_s (k ) - X_s^{[F]} (k)| \le \delta_k (s )  $ (see \eqref{Eq:A6}), we obtain 
\begin{equation}\label{eq:convolution} \delta (t+1) \le H * \delta (t +1) + \gamma  \Sigma_i(F)  g * 1 (t+1)  e_i,  
\end{equation}
where $e_i $ is the $i-$the unit vector. In the above formula, 
$$ ( H * \delta (t ))_j = \sum_{k\in I} \sum_{s=0}^t H_{j,k } (t- s) \delta_k (s) $$ 
is the operator convolution product, and the inequality in \eqref{eq:convolution} has to be understood coordinate-wise.  

Now let $\alpha_0$ be as in \eqref{eq:alpha_0} 
and introduce $\tilde H (t) = e^{-  \alpha_0    t } H (t) ,$ $ \tilde \delta (t) = e^{- \alpha_0  t }  \delta (t  )$ and $\tilde 1 (t) = e^{ - \alpha_0  t} .$ Multiplying the above inequality with $ e^{- \alpha_0 ( t+1) }, $ we obtain
$$ \tilde \delta (t+1 ) \le  \tilde H * \tilde \delta (t +1) + \tilde g * \tilde 1 (t+1 )  \gamma  \Sigma_i(F) e_i .$$
Let $ \| \tilde \delta  \|_1  = ( \| \tilde \delta_i  \|_1, i \in I)  $ be the column vector where each entry is given by $\| \tilde \delta_i  \|_1 = \sum_{t=0}^\infty \delta_i (t) .$  Then we obtain,  summing over $t  \geq 0, $ 
$$
 \| \tilde \delta  \|_1 
\le \Lambda ( \alpha_0)  \| \tilde \delta  \|_1 + \frac{1}{1 - e^{ - \alpha_0 }} \| \tilde g \|_1    \Sigma_i(F)   e_i ,\\
$$
implying that 
\begin{equation}\label{eq:A11}
 ( Id -\Lambda ( \alpha_0)   )\| \tilde \delta  \|_1 \le  \frac{1}{1 - e^{ - \alpha_0 }} \| \tilde g \|_1    \Sigma_i(F) e_i  .
\end{equation}

By \eqref{eq:alpha_0}, $Id -\Lambda ( \alpha_0)$ is invertible, and it is well-known that the operator norm of the inverse is bounded by 
$$ \| \!| ( Id - \Lambda ( \alpha_0))^{-1} \| \!| \le ( 1 - \| \!| \Lambda ( \alpha_0) \| \!| )^{-1} = C( \alpha_0) .$$
Moreover, $(Id -\Lambda ( \alpha_0))^{-1} : \ell_\infty^+ \to \ell_\infty^+ , $ where $ \ell_\infty^+ = \{ ( \xi_j)_{j\in I} :  \xi_j \geq 0  \} .$ Therefore, \eqref{eq:A11} implies 
\begin{equation}
\sup_{j\in I} \| \tilde \delta_j  \|_1 \le \left[ \frac{1}{1 - e^{ - \alpha_0 }} \| \tilde g \|_1  \Sigma_i(F) \right] C (\alpha_0)  .
\end{equation}
By using the union bound and  \eqref{Eq:A6}, it follows that
$$
  \sup_{j\in I } \P_{x} ( \exists s \in [1, t ] : X_s (j ) \neq X_s^{[F]} (j)  )  \le \sup_{j\in I} 
  \sum_{ s=1}^t   \delta_j (s ) \le  e^{ \alpha_0t } \sup_{j\in I} \| \tilde \delta_{j}  \|_1,
$$
which implies the assertion of Item 1.

The proof of Item 2 is similar to the above argument, except that now it is possible to work directly with $ g(t ) $ instead of $\tilde g (t) .$ In this case, we write simply $ \bar \delta (t) = \sup_{j\in I} \delta_j ( t) .$ \eqref{eq:convolution} implies that 
$$ (\sup_{0 \le s \le t }\bar \delta ( s)) \le \chi \, (\sup_{0 \le s \le t }\bar \delta (s))  + \gamma  \varrho \sum_{ k \notin V_i \cap F  } |W_{k \to i } |  , $$
which implies the assertion.
 
\end{proof}


\begin{proof}[Proof of \eqref{eq:discrepancy}]
The coupling inequality \eqref{eq:discrepancy} follows now directly from \eqref{eq:discrepancy3} (\eqref{eq:discrepancy4}, respectively).
\end{proof}

%

\section*{Acknowledgments}

This work is part of USP project {\em Mathematics, computation, language
and the brain}, FAPESP project {\em Research, Innovation and
Dissemination Center for Neuromathematics} (grant 2013/07699-0), CNPq projects {\em Stochastic modeling of the brain activity}
(grant 480108/2012-9) and {\em Plasticity in the brain after a brachial plexus lesion} (grant 478537/2012-3), and of the project Labex MME-DII (ANR11-LBX-0023-01).

AD and GO are fully supported by a FAPESP fellowship (grants  2016/17791-9 and 2016/17789-4 respectively). AG is partially supported  by CNPq fellowship (grant 309501/2011-3.) 

We thank the anonymous reviewers for their valuable comments and suggestions which helped us to improve the paper. We warmly thank B. Lindner and A.C. Roque  for indicating us important references concerning stochastic models for neuronal activity.  
\bibliography{Bibli}{}
\bibliographystyle{imsart-nameyear}
\nocite{GalEva:13}
\nocite{Errico:14}
\nocite{AG:14}
\nocite{PierreEva:14}
\nocite{Evafou:14}
\nocite{EvaGalves:15}
\nocite{Duarte_Ost:15}
\nocite{Karina:15}
\nocite{RobTou:14}
\nocite{Montanari:09}
\nocite{Bresler:15}
\nocite{GalOrlTak:15}
\nocite{BreElcSly:08}
\nocite{LochOrla:11}
\nocite{lerTak:11}
\nocite{lerTak:16}
\nocite{ravikumar:10}
\nocite{reynaud:13}
\nocite{brillinger:88}
\nocite{GalvesLeonardi:08}
\end{document}